\documentclass[12pt]{amsart}

\usepackage{amstext}
\usepackage{amsthm}
\usepackage{amsmath}
\usepackage{amssymb}
\usepackage{latexsym}
\usepackage{amsfonts}
\usepackage{graphicx}

\usepackage[pagebackref,hypertexnames=false, colorlinks, citecolor=red, linkcolor=red]{hyperref} 
\usepackage{amsrefs}

\input txdtools

\begin{document}

\newcommand{\ci}[1]{_{ {}_{\scriptstyle #1}}}
\newcommand{\ti}[1]{_{\scriptstyle \text{\rm #1}}}
\newcommand{\ut}[1]{^{\scriptstyle \text{\rm #1}}}

\newcommand{\fdot}{\,\cdot\,}

\newcommand{\norm}[1]{\ensuremath{\left\|#1\right\|}}
\newcommand{\abs}[1]{\ensuremath{\left\vert#1\right\vert}}
\newcommand{\p}{\ensuremath{\partial}}
\newcommand{\pr}{\mathcal{P}}

\newcommand{\pbar}{\ensuremath{\bar{\partial}}}
\newcommand{\db}{\overline\partial}
\newcommand{\D}{\mathbb{D}}
\newcommand{\B}{\mathbb{B}}
\newcommand{\Sp}{\mathbb{S}}
\newcommand{\T}{\mathbb{T}}
\newcommand{\R}{\mathbb{R}}
\newcommand{\Z}{\mathbb{Z}}
\newcommand{\C}{\mathbb{C}}
\newcommand{\N}{\mathbb{N}}
\newcommand{\scrH}{\mathcal{H}}
\newcommand{\scrL}{\mathcal{L}}
\newcommand{\td}{\widetilde\Delta}
\newcommand{\wt}{\widetilde}

\newcommand{\e}{\varepsilon}

\newcommand{\cM}{\mathcal{M}}
\newcommand{\cX}{\mathcal{X}}
\newcommand{\cY}{\mathcal{Y}}

\newcommand{\La}{\langle }
\newcommand{\Ra}{\rangle }
\newcommand{\rk}{\operatorname{rk}}
\newcommand{\card}{\operatorname{card}}
\newcommand{\ran}{\operatorname{Ran}}
\newcommand{\osc}{\operatorname{OSC}}
\newcommand{\im}{\operatorname{Im}}
\newcommand{\re}{\operatorname{Re}}
\newcommand{\tr}{\operatorname{tr}}
\newcommand{\dist}{\operatorname{dist}}
\newcommand{\supp}{\operatorname{supp}}

\newcommand{\fA}{\mathfrak{A}}
\newcommand{\fB}{\mathfrak{B}}

\newcommand{\vf}{\varphi}
\newcommand{\f}[2]{\ensuremath{\frac{#1}{#2}}}

\newcommand{\clos}{\operatorname{Clos}}
\newcommand{\inter}{\operatorname{Int}}


\newcommand{\entrylabel}[1]{\mbox{#1}\hfill}

\newenvironment{entry}
{\begin{list}{X}%
  {\renewcommand{\makelabel}{\entrylabel}%
      \setlength{\labelwidth}{55pt}%
      \setlength{\leftmargin}{\labelwidth}
      \addtolength{\leftmargin}{\labelsep}%
   }%
}%
{\end{list}}


\numberwithin{equation}{section}

\newtheorem{thm}{Theorem}[section]
\newtheorem{lm}[thm]{Lemma}
\newtheorem{cor}[thm]{Corollary}
\newtheorem{conj}[thm]{Conjecture}
\newtheorem{prob}[thm]{Problem}
\newtheorem{prop}[thm]{Proposition}
\newtheorem*{prop*}{Proposition}

\theoremstyle{definition}
\newtheorem{df}[thm]{Definition}
\newtheorem*{df*}{Definition}

\theoremstyle{remark}
\newtheorem{rem}[thm]{Remark}
\newtheorem*{rem*}{Remark}

\title{Corona Solutions Depending Smoothly on Corona Data}

\author[S. Treil]{Sergei Treil$^\dagger$}
\address{Sergei Treil, Department of Mathematics\\ Brown University\\ 151 Thayer Street\\ Providence, RI USA 02906}
\email{treil@math.brown.edu}
\thanks{Research supported in part by a National Science Foundation DMS grant \#0800876.}

\author[B. D. Wick]{Brett D. Wick$^\ddagger$}
\address{Brett D. Wick, School of Mathematics\\ Georgia Institute of Technology\\ 686 Cherry Street\\ Atlanta, GA USA 30332-0160}
\email{wick@math.gatech.edu}
\thanks{Research supported in part by a National Science Foundation DMS grants \# 1001098 and \# 0955432.}

\thanks{This note began at the workshop ``The Corona Problem: Connections Between Operator Theory, Function Theory and Geometry'''.  The authors thank the Fields Institute for the pleasant working conditions.}


\begin{abstract} 
In this note we show that if the Corona data depends continuously (smoothly) on a parameter, the solutions of the corresponding Bezout equations can be chosen to have the same smoothness in the parameter. 
\end{abstract}

\maketitle

\section{Introduction and Main Results}

\subsection{Background}
Let $H^\infty:=H^\infty\left(\D\right)$ denote the collection of all bounded analytic functions on the unit disc $\D$.  The classical Carleson Corona Theorem, see \cite{Carl}, states that 
if functions
\(f_{j}\in H^{\infty}\) are such that
\(\sum_{j=1}^{N}\abs{f_{j}}^{2}\geq\delta^2>0\), then  there exists
functions $g_{j}\in H^{\infty}$ such that
$\sum_{j=1}^{N}g_{j}f_{j}=1$.   This is equivalent to the fact that the unit disc $\D$ is dense in the
maximal ideal space of the algebra $H^\infty$, but the importance of 
the
Corona Theorem goes much beyond the theory of maximal ideals of
$H^\infty$.  The Corona Theorem, and especially its generalization, the so called
Matrix (Operator) Corona Theorem play an important role in 
operator
theory (such as the angles between invariant subspaces, unconditionally convergent
spectral decompositions, computation of spectrum, etc.).

The Vector Valued Corona Theorem says that if  $f=(f_k)_{k=1}^\infty\in H^\infty_{\ell^2}$ is a bounded 
analytic function whose values are in $\ell^2$ such that 
\begin{equation}
\label{C}
\tag{C}
1\ge \sum_{k=1}^\infty \abs{f_k(z)}^2\ge\delta^2  >0, \qquad \forall z\in\D, 
\end{equation}
then there exists $g\in H^\infty_{\ell^2}$ solving the Bezout equation, 
\begin{equation}
\label{B}
\tag{B}
g^T(z)f(z):=\sum_{k=1}^\infty g_k(z) f_k(z)\equiv 1\quad\forall z\in\D.
\end{equation}
Moreover, the function $g$ can be chosen to satisfy (for  $\delta\le1/2$)
\begin{align}
\label{CorEst-01}
\|g(z)\|\ci{\ell^2} \le C\frac{1}{\delta^2}\log\frac1{\delta} \qquad \forall z\in \D,
\end{align}
with $C$ an absolute constant.  Note also that the above condition \eqref{C} is necessary for the existence of the function $g$ appearing in \eqref{B}.  Condition \eqref{C} is usually called the  Carleson Corona Condition (or simply the Corona Condition).  Condition \eqref{B} is named so after the Bezout equation.

The main result of this paper is that if the Corona data $f$ depends smoothly on a parameter $s$, then one can find a solution $g$ also depending smoothly on the parameter $s$.



\subsection{Preliminaries} 
To give the precise statements we  need to introduce some notation. 
\label{s:prelim}
\subsubsection{Function spaces \(C^r\left(K;H^\infty\right) \) and \(H^\infty\left(C^r(K)\right)\)}
Let $K$ be a compact Hausdorff space. We say that a function $f:\D\times K\to \C$ belongs to $C\left(K;H^\infty\right)$ if $f(\fdot , s)\in H^\infty$ for any $s\in K$ and the function $s\mapsto f (\fdot, s)$ is a continuous function on $K$ (with values in the Banach space $H^\infty$). Since $K$ is compact, the above continuity simply means that the function $f$ is uniformly continuous in the variable $s\in K$ on $\D\times K$. Thus the space $C\left(K;H^\infty\right)$ can be identified with the space of functions $f:\D\times K\to \C$ which are analytic in the variable $z\in\D$ and uniformly continuous on $\D\times K$ in the variable $s\in K$.  In other words, for any $\e>0$ there exists $\delta>0$ that for all $z\in \D$ and for all $s, s'\in K$ the inequality $|f(z, s)-f(z, s')|<\e$ holds whenever $|s-s'|<\delta$.

We can also define the space $H^\infty\left(C(K)\right)$ as the space of functions $f:\D\times K\to \C$ such that the function $z\mapsto f(z, \fdot)$ is a bounded analytic function with values in $C(K)$. Since for functions with values in a Banach space weak and strong analyticity coincide, see  for example \cite{Nik-book-v1}*{Theorem 3.11.4}, the space $H^\infty\left(C(K)\right)$ can be identified with the collection of bounded functions $f:\D\times K\to \C$ which are analytic in the variable $z\in \D$ and continuous in the variable $s\in K$.

If \(K\) has a smooth structure, for example if \(K\) is a compact subset of $\R^d$ or a smooth manifold such that $K =\clos \left(\inter K\right)$, we can define smooth versions of these function spaces.  The spaces $C^r\left(K;H^\infty\right)$, $r\in \N\cup\infty$ is the space of functions $f:\D\times K\to \C$ such that the vector-valued function $s\mapsto f(\fdot, s)\in H^\infty$  has all its partial derivatives (in the strong sense) up to order $r$ existing on $\inter K$  and are continuous up to the boundary.  We will identify the function with the derivative of order $0$, so the function itself is continuous up to the boundary. Similarly, one can say that $C^r\left(H^\infty\right)$ consists of functions $f:\D\times K\to \C$ which are analytic in the variable $z\in \D$ and all the partial derivatives in the variable $s\in K$ up to order $r$ exist and are uniformly continuous in $s$ on $\D\times \inter K$.

We will also consider vector-valued versions $C^r \left(K; H^\infty_{\ell^2}\right)$ of these spaces, consisting of all functions $f : \D\times K \to \ell^2$ such that the function $s \mapsto f(\fdot, s) $ is a $C^r$ function with values in the vector-valued space $H^\infty_{\ell^2}$, i.e. $H^\infty$ with values in  $\ell^2$.

We can also define the space of functions $H^\infty(C^r(K))$, which is the space of functions $f:\D\times K\to \C$ which are analytic in the variable $z\in \D$ and such that all partial derivatives in $s$ up to the order $r$ exist  on $\D\times \inter K$, are  continuous and bounded there, and for all $z\in \D$ extend continuously to the boundary of $K$.   Again, in a similar manner, we also consider the vector-valued version $H^\infty \left(C^r_{\ell^2}(K)\right) $ of such spaces, consisting of all functions $f :\D\times K \to \ell^2$ such that the function $z\mapsto f(z, \fdot)$ is a bounded analytic function with values in the vector-valued space $C^r_{\ell^2}(K)$, i.e., $C^r(K)$ with values in $\ell^2$.

Note, that the only difference between the spaces $C^r\left(K;H^\infty\right)$ and $H^\infty\left(C^r(K)\right)$ is that in the former we require the uniform continuity in $s$ on $\D\times K$ of the derivatives in $s$ of order $r$, and in the latter not.

\subsubsection{Domains with uniformly continuous path metric}
\label{s:PathMetrUinfCont}
For an open connected subset $\Omega$ of a smooth manifold in addition to the metric $\rho$ inherited from the manifold we can define the \emph{path metric} $\rho\ti p(x,y)$ as the infimum of the lengths of the paths in $\Omega$ connecting $x$ and $y$. Clearly $\rho\le \rho\ti p$, and $\rho\ti p$ is continuous with respect to $\rho$.

We are interested in domains where the path metric is \emph{uniformly continuous}, because   for such domains 
$ H^\infty\left(C^r(K)\right)\subset C^{r-1}\left(K;H^\infty\right)$.   An example of a domain with uniformly continuous path metric is the so-called weakly uniform domains. 
We say that a domain $\Omega$ (in $\R^d$ or in a smooth manifold) is \emph{weakly uniform}  if there exists $M<\infty$ such that  any two points $x_1, x_2\in \Omega$ can be connected by a smooth path in $\Omega$ of length at most $M\dist(x_1, x_2)$. 
For example, a bounded domain with a smooth boundary is definitely a weakly uniform domain.  If $\inter K$ is a weakly uniform domain, then an easy application of the Mean Value Theorem shows that $ H^\infty\left(C^r(K)\right)\subset C^{r-1}\left(K;H^\infty\right)$.

\subsection{Main Results}
The first theorem gives us smooth dependence of the solution of the Corona Bezout equation on a parameter.  
If $r\ge 1$ we assume that  $K$ is a compact subset of $\R^d$ or of a smooth manifold, such that $K=\clos\left(\inter K\right)$. For $r=0$ we only assume that $K$ is a compact Hausdorff space. 
\begin{thm}
\label{t:main-01}
Let $r\in\Z_+$ and let $f=(f_k)_{k=1}^\infty\in C^r \left(K;H^\infty_{\ell^2}\right)$  be such that 
\[
1\ge \sum_{k=1}^\infty \left|f_k(z, s)\right|^2\ge\delta^2 >0\qquad \forall (z,s)\in \D\times K. 
\] 
Then there exist $g=(g_k)_{k=1}^\infty\in C^r\left(K;H^\infty_{\ell^2}\right)$ such that 
\[
g^T(z,s)f(z,s) := \sum_{k=1}^\infty f_k(z,s) g_k(z,s) \equiv 1\qquad \forall (z,s)\in \D\times K. 
\]
Moreover,
\begin{align}
\label{SolEst-01}
\left\|g(z, s)\right\|_{\ell^2}  & := \left(\sum_{k=1}^\infty |g_k(z, s)|^2 \right)^{1/2}  \le 2 C(\delta) <\infty,  \\
\intertext{where $C(\delta) $ is the estimate in the Corona Theorem for $H^\infty$, and}
\label{SolEst-02}
\left\|g\right\|\ci{C^{\alpha}\left(K;H^\infty_{\ell^2}\right)} & \le C\left(\alpha, \omega_f^{-1}\left(\frac{1}{2C(\delta)}\right),  K\right) \left\|f\right\|\ci{C^{\alpha}\left(K;H^\infty_{\ell^2}\right)} \qquad \forall \alpha\in\Z_+, \ \alpha \le r;
\end{align}
here $\omega_f$ is the modulus of continuity of $f$, and \(\omega_f^{-1}\) is its inverse. 
\end{thm}
Note that for $r=0$ the estimate \eqref{SolEst-02} is superfluous since it follows immediately from \eqref{SolEst-01}.

The next theorem can be interpreted as the Corona Theorem for the algebra $H^\infty\left(C^r(K)\right)$, $r\ge 1$. We assume here that $K$ is a subset of a smooth manifold, $K=\clos\left(\inter K\right)$ such that the path metric on $K$ is uniformly continuous (with the respect to the metric inherited from the manifold). 

\begin{thm}
\label{t:main-02}
Let $r\in\N\cup\infty$ and let $f =(f_k)_{k=1}^\infty \in H^\infty\left(C^r_{\ell^2}(K)\right)$ be such that 
\[
1\ge \sum_{k=1}^\infty \left|f_k(z, s)\right|^2\ge\delta^2 >0\qquad \forall (z,s)\in \D\times K. 
\] 
Then there exist $g=(g_k)_{k=1}^\infty \in H^\infty\left(C^r_{\ell^2}(K)\right)$ such that 
\[
g^T(z,s) f(z,s) := \sum_{k=1}^\infty f_k(z,s) g_k(z,s) \equiv 1\qquad \forall (z,s)\in \D\times K. 
\]
Moreover,
\begin{align}
\label{SolEst-03}
\left\|g(z, s)\right\|_{\ell^2} 
& \le 2C(\delta) <\infty \qquad \forall (z,s)\in \D\times K,  \\
\intertext{where $C(\delta) $ is the estimate in the Corona Theorem for $H^\infty$, and}
\label{SolEst-04}
\left\|g\right\|\ci{H^\infty\left(C_{\ell^2}^\alpha(K)\right)}  
& \le C\left(\delta, \alpha, K,  \left\| f\right\|\ci{C^1_{\ell^2}(K)}\right)  \left\|f\right\|\ci{H^\infty\left(C_{\ell^2}^\alpha(K)\right)} \qquad \forall \alpha\in\Z_+, \ \alpha \le r.
\end{align}
\end{thm}
\section{The Proofs}
\label{s:proof}

\begin{proof}[Proof of Theorem \ref{t:main-01}]
Let $f$ and $g$ be the column vectors with entries $f_k$ and $g_k$ respectively. By the Carleson Corona Theorem for infinitely many functions there exists a constant $C(\delta)$ such that for any $f\in H^\infty_{\ell^2}$ satisfying 
\[
1\ge \left|f(z)\right|\ge\delta>0\quad\forall z\in\D,
\]
there exists $g\in H^\infty_{\ell^2}$, $\left\|g\right\|\ci{H^\infty_{\ell^2}} \le C(\delta)$ such that $g^T(z) f(z) \equiv 1$ for all $z\in \D$.

Here is the main idea of how the proof will proceed.  We will use this above for each function $f(\fdot, s)$.  However, since $s\mapsto f(\fdot, s)$ is continuous we will be able to construct a perturbation of the solutions that remains close to $1$ on some open neighborhood.  Then, using a partition of unity argument, we can construct the desired function, which is again close to $1$, but now for all $s$ and all $z$.

Since $s\mapsto f(\fdot, s)$ is continuous, for each point $s\in K$ let $U_s$ be a neighborhood of $s$ such that for all $s'\in U_s$ 
\[
\left\|f(\fdot, s) - f(\fdot, s')\right\|\ci{H^\infty_{\ell^2}} \le \frac{1}{2C(\delta)}.
\] 
If $g_s(\fdot) \in H^\infty_{\ell^2}$ solves the Bezout equation $g_s^T(\fdot) f(\fdot, s)\equiv 1$ satisfying $\|g_s(\fdot)\|\ci{H^\infty_{\ell^2}} \le C(\delta)$, then clearly for $s'\in U_s$ 
\begin{align}
\label{pert-01}
\left\| 1-g_s^T (\fdot) f(\fdot , s') \right\|_{H^\infty} \le \frac{1}{2}.
\end{align}

Since, by assumption, $K$ is compact, we can take a finite cover $U_k := U_{s_k}$, $k=1, 2, \ldots, N$ such that $K \subset \bigcup_{k=1}^N U_k$, and let $\eta_k$ be a partition of unity subordinated to this covering. This means that $0\le \eta_k\le 1$, $\supp \eta_k \subset U_k$ and $\sum_1^N \eta_k(s) \equiv 1$.   Note, that if $K$ has a smooth structure then one can take $\eta_k\in C^\infty(K)$. Moreover, it is possible to construct a subcover $U_{k}$ and the $C^\infty(K)$ partition of unity $\eta_k$ such that for all $\alpha\in \N$, 
\begin{align}
\label{part-est-01}
\sum_{k=1}^N \left\| \eta_k\right\|\ci{C^\alpha(K)} \le C\left(\alpha, \omega_f\left(\frac{1}{2C(\delta)}\right), K\right)
\end{align}

For each $k$ let $ g_{s_k}(\fdot)$ be a solution of the Bezout equation 
\[
 g_{s_k}^T (\fdot) f(\fdot, s_k) \equiv 1
\]
satisfying $\left\|g_{s_k}(\fdot)\right\|\ci{H^\infty_{\ell^2}}\le C(\delta)$. Define
\[
 \wt g(z, s):= \sum_{k=1}^N \eta_k(s) g_{s_k}(z). 
\]
Then \eqref{pert-01} clearly implies that for all $s'\in K$
\begin{align}
\label{pert-02}
\left\|  1 - \wt g^T(\fdot, s') f(\fdot, s') \right\|_{H^\infty}\le \frac{1}{2}. 
\end{align}
Indeed, for $k$ such that $\eta_k(s')>0$ we have $s'\in U_{s_k}$, so the estimate \eqref{pert-01} holds  for $s=s_k$. 
But for any $s'$ the function  $\wt g(\fdot, s')$ is a convex combination of such $g_{s_k}(\fdot)$, so we get \eqref{pert-02} as a convex combination of estimates \eqref{pert-01}.

Inequality \eqref{pert-02} implies that the (scalar) function $\wt g^T(\fdot, s') f(\fdot, s')$ is invertible in $H^\infty$ and that $\left\| ( \wt g^T(\fdot, s') f(\fdot, s'))^{-1} \right\|_{H^\infty} \le 2$. Therefore the function $ g := (\wt g^T f)^{-1} \wt g$ solves the Bezout equation $g^T f\equiv 1$ and satisfies \eqref{SolEst-01}. Computing the derivatives of $g$ in $s$ and taking into account \eqref{part-est-01} we easily get \eqref{SolEst-02}. 
\end{proof}

\begin{proof}[Proof of Theorem \ref{t:main-02}]
As we mentioned above in Section \ref{s:PathMetrUinfCont}, if the path metric $\rho\ti p$ is uniformly continuous 
then 
\[
H^\infty\left(C^r(K)\right)\subset C^{r-1}\left(K;H^\infty\right),  
\] 
and the same for the $\ell^2$-valued case. In particular, $ H^\infty\left(C^r_{\ell^2}(K)\right)\subset C\left(K;H^\infty_{\ell^2}\right)$ for all $r\ge 1$.

Since $s\mapsto f(\fdot, s)$ is continuous, we can construct the partition of unity $\eta_k$ as in the proof of Theorem \ref{t:main-01} above. Note that if the path metric $\rho\ti p$ is uniformly continuous, the modulus of continuity $\omega_f$ can be estimated  via the modulus of continuity of $\rho\ti p$ and 
\[
\omega_f \left(\frac{1}{2C(\delta)}\right) \le 
\omega_{\rho\ti p}\left(\frac{1}{2C(\delta)}\right)  \left\|f\right\|\ci{C^1_{\ell^2}}.
\] 
where $\omega_{\rho\ti p}$ is the modulus of continuity of $\rho\ti p$.  The rest follows as in the proof of Theorem \ref{t:main-01}. 
\end{proof}

\section{Some Remarks}
The method of the proof allows us to get smooth dependence on the parameter for the Corona type problems in  very general situations, like in the matrix and operator Corona Problem, as well as in the Corona Problem for multipliers of the reproducing kernel Hilbert spaces.  Below, we state some sample results. The proofs are similar to what was presented in Section \ref{s:proof} and are left as an exercise for the reader.  

Let $X$ and $Y$ be Banach spaces, and let $\fA=\fA\left(\Omega; B(X,Y)\right)$ and $\fB= \fB\left(\Omega;B(Y,X)\right)$ be some Banach spaces of functions  on some set $\Omega$ with values in $B(X,Y)$ (the space of bounded operators from $X$ to $Y$) and in $B(Y,X)$ respectively.

As in Section \ref{s:prelim} we can introduce the spaces $C^r\left(K;\fA\right)$, $C^r\left(K;\fB\right)$, where $K$ is a compact set.  The space $C^r\left(K;\fA\right)$ consists of all functions $f:\Omega\times K \to B(X, Y)$ such that the function $s\mapsto f(\fdot, s)$ is an $r$ times continuously differentiable function on $K$, and similarly for $C^r\left(K;\fB\right)$. As in Section \ref{s:prelim} we assume that  $K$ is a Hausdorff compact space if $r=0$, and if $r\ge1$ we assume that $K$ is a compact subset of a smooth manifold such that $K= \clos\left(\inter K\right)$.

We also assume that \(\fB\fA\fB \subset \fB\), i.e.~that \(\fB\fA\) belongs to the multiplier algebra of \(\fB\). This implies that for \(f\in\fA\) and \(g, h\in \fB\)
\begin{align}
\label{mult-01}
\left\| g f h \right\|\ci{\fB} 
\le C(\fA, \fB) \left\|g\right\|\ci{\fB} \left\|f\right\|\ci{\fA} \left\|h\right\|\ci{\fB}  .
\end{align}

A typical example here would be multiplier spaces of (vector-valued) function spaces. Let for example \(\cX\) and \(\cY\) be  spaces of functions  on \(\Omega\) with values in \(X\) and \(Y\) respectively. Define  \(\fA\) and \(\fB\) as the multiplier function spaces between \(\cX\) and \(\cY\). Namely, \( \fA \) is the space of operator-valued functions \( \vf :\Omega \to B(X, Y) \) such that the map \( f\mapsto \vf f\) is a bounded operator from \(\cX\) to \(\cY\); similarly, the space \(\fB\) will be the collection of functions \(\vf:\Omega \to B(Y,X) \) such that the map \( f\mapsto \vf f\) is a bounded operator from  \(\cY\) to \(\cX\).  

For a concrete example, if \(\cX=H^2(X)\) and \(\cY= H^2(Y)\) then \(\fA = H^\infty \left(B(X,Y)\right) \) and \(\fB = H^\infty\left(B(Y, X)\right)\) and we are in the situation of the operator corona. This situation with \(X=\ell^2\) and \(Y = \C\) gives us the settings of Theorem \ref{t:main-01}.   Another example would be the spaces of multipliers of Bergman or Dirichlet type spaces (in the disc or more general domains in \(\C\) or \(\C^n\)). But in general we do not need to assume that the spaces \(\fA\) and \(\fB\) are multiplier spaces.

\begin{thm}
\label{t:main-03}
Under the above assumptions, let \(f\in C^r\left(K;\fA\right) \) be such that \(\|f(\fdot,s)\|\ci{\fA}\le 1\) for all \(s\in K\).  Suppose that for all \(s\in K\) there exists \(g_s\in \fB\) such that 
\[
g_s (z) f(z, s) \equiv I\quad\forall z\in\Omega, \qquad \|g_s\|\ci{\fB}\le C_0. 
\]
Then there exists \(g\in C^r(K;\fA) \) such that 
\[
g(z, s) f(z, s)\equiv I\quad \forall(z,s)\in\Omega\times K, \qquad \text{and} \qquad \| g(\fdot, s)\|\ci{\fB} \le 2C_0. 
\]
Moreover, for all \(\alpha\in \Z_+\), \(\alpha\le r\) we have the estimate 
\[
\| g\|\ci{C^\alpha\left(K;\fB\right)} \le  C \left( \alpha, \omega_f^{-1}\left(\frac{1}{2C_0 C(\fA, \fB)}\right) , K \right) 
\|f\|\ci{C^\alpha\left(K, \fB\right)}.
\]
Here, again, \(\omega_f\) is the modulus of continuity of \(f\) and \( C(\fA, \fB) \) is the constant from \eqref{mult-01}. 
\end{thm}

\begin{rem*}
Note that in the above theorem we do not assume any kind of ``Corona Theorem'' for the spaces \(\fA\) and \(\fB\); we just postulate the solvability of the Bezout equations for each \(s\in K\). 
\end{rem*}

The proof of Theorem \ref{t:main-03} goes along the lines of the proof of Theorem \ref{t:main-01}. We construct a cover \(U_s\), \(s\in K\) of \(K\) such that 
\[
\left\|f(\fdot, s) - f(\fdot, s')\right\|\ci{\fA} \le \frac{1}{2C_0C(\fA, \fB)} \qquad \forall s'\in U_s, 
\] 
take a finite subcover and construct a partition of unity \(\eta_k=\eta_{s_k}\), then finally construct the function \(\wt g\), etc. 

Estimate \eqref{mult-01} implies that instead of \eqref{pert-02} we have the estimate
\[
\| I - \wt g(\fdot, s') f(\fdot, s') \|\ci{\cM(\fA)} \le \frac12, 
\]
where \(\cM(\fA)\) is the (left) multiplier algebra of \(\fA\). But that implies that 
\[
\| \left( \wt g(\fdot, s') f(\fdot, s') \right)^{-1} \|\ci{\cM(\fA)} \le 2
\]
so all the estimates in Theorem \ref{t:main-03} follow.



\end{document}